\documentclass[11pt]{amsart}
\usepackage{amsfonts}
\usepackage{graphicx,color}
\usepackage{amsthm}
\usepackage{amssymb,amsmath}
\usepackage{mathtools}

\date{\today}

\newtheorem{theorem}{Theorem}[section]

\newtheorem{conjecture}[theorem]{Conjecture}

\theoremstyle{definition}

\newtheorem{remark}[theorem]{Remark}

\title[3-distance 5-designs]{\bf On 3-distance spherical 5-designs}

\author[P. Boyvalenkov]{Peter Boyvalenkov$^\dagger$}
\address{Institute of Mathematics and Informatics, Bulgarian Academy of Sciences,
8 G Bonchev Str.,
1113  Sofia, Bulgaria}
\email{peter@math.bas.bg}
\thanks{\noindent $^\dagger$ The research of the first author was supported, in part, by Bulgarian NSF under project KP-06-N32/2-2019. }
\author[N. Safaei]{Navid Safaei$^{\dagger}$}
\address{Research Institute of  Policy Making, Sharif University of Technology,
Tehran, Iran}
\email{navid\_safaei@gsme.sharif.edu}

\begin{document}
\maketitle

\begin{abstract}
Inspired by a recently formulated conjecture by Bannai et al. we investigate spherical codes which admit exactly three different 
distances and are spherical 5-designs. 
Computing and analyzing distance distributions we provide new proof of the fact (due to Levenshtein) that such codes are maximal and 
rule out certain cases towards a proof of the conjecture. 
\end{abstract}
%
%

\section{Introduction}

Recently, Bannai et al. \cite{BBXYZ} formulated the following conjecture. 

\begin{conjecture} \label{cb}
If $C \subset \mathbb{S}^{n-1}$ is a spherical 3-distance 5-design, then one of the following holds.

(1) $n=2$ and $C$ is a regular hexagon or a regular heptagon.

(2) $C$ is a tight spherical 5-design.

(3) $C$ is derived from a tight spherical 7-design.
\end{conjecture}

In fact, in \cite{BBXYZ} Conjecture \ref{cb} was referred to as "Many people have conjectured, but not written down, the classification of 3-distance 5-designs". Indeed, it probably goes back to 1977 when Delsarte, Goethals and Seidel \cite{DGS} introduced spherical designs in a seminal 
paper. The first author was among these who considered similar questions in 1990's (see \cite{BD,BDL}) and is aware about 
the conjecture since then. 

It is known that the cardinality of any spherical 3-distance 5-design $C \subset \mathbb{S}^{n-1}$ satisfies
\[ n(n+1) \leq |C| \leq {n+2 \choose 3}+{n+1 \choose 2}=\frac{n(n+1)(n+5)}{6}. \]
Both bounds are due to Delsarte-Goethals-Seidel \cite{DGS}. The lower bound follows since $C$ is a spherical
5-design, the upper bound -- since $C$ is a spherical 3-distance set (for upper bounds for few distance sets 
see also \cite{BBS,BKNS,GY,HR,PP}). The main result in Levenshtein \cite{Lev92} (see also \cite[Section 5]{Lev98}) 
implies that any spherical 3-distance 5-design is a maximal code. 

In this paper we present new proof of the Levenshtein's result from \cite{Lev92} and derive necessary conditions
which underline a possible way to the complete proof of Conjecture \ref{cb}. To this end we compute and carefully investigate 
the distance distributions of spherical 3-distance 5-designs. We prove that twice the cardinality is divisible by the dimension $n$, thus reducing 
by a factor of $n$ the number of cases under suspicion. 

In Section 2 we prove that every spherical 3-distance 5-design needs to have its three inner products exactly equal to the 
zeros of certain polynomial used by Levenshtein for derivation of his bounds. Section 3 is devoted to computing the distance 
distribution of corresponding codes and different presentations of the results. In Section 4 we obtain divisibility conditions
which rule out significant amount of cases. Section 5 describes derived codes as a possible continuation of our investigation. 
In Section 6 we explain our computer investigation which verifies Conjecture \ref{cb} in all dimensions $n \leq 1000$. 

\section{Inner products}

Let $C \subset \mathbb{S}^{n-1}$ be a spherical 3-distance 5-design of cardinality $|C|=M$ and 
inner products $a<b<c$. The case $n=2$ is elementary. If $n \geq 3$ and $C$ is a tight 5-design, then $a=-1$, $c=-b=1/m$, where $m$ is odd positive integer,
such that $n=m^2-2$ or $m=1/\sqrt{5}$ and $C$ is the icosahedron in three dimensions. Examples are known 
for $m=3$ and $5$ only, and another folklore conjecture says that no other tight 5-designs exist. 
Thus, in what follows we focus in dimensions $n \geq 3$ and assume that $C$ is not a tight spherical 5-designs. 
Conjecture \ref{cb} is now equivalent to prove that always (3) happens. In this case, $C$ has dimension $n=3m^2-5$, cardinality $M=m^4(3m^2-5)/2$
and inner products $-1/(m-1)$, $-1/(m^2-1)$, and $1/(m+1)$, where $m \geq 2$ is a positive integer. Examples are 
known for $m=2$ and $3$ only. 

The main result from \cite{Lev92}
implies that $|a|>|c|>|b|>0$ (see \cite[Corollary 3.9]{BD} for detailed proof of these inequalities), and
$a$ and $b$ are the roots of the quadratic equation
\[ (n + 2)[(n + 2)c^2 + 2c-1]t^2 + 2c(c + 1)(n + 2)t + 3-(n + 2)c^2 = 0, \]
where $c$ satisfies the equation
\begin{equation} \label{abc-equ}
M=\frac{n((n+2)(n+3)c^2+4(n+2)c-n+1)(1-c)}{2c(3-(n+2)c^2)}.
\end{equation}
In the terminology of \cite{Lev92}, one has $a=\alpha_0$, $b=\alpha_1$, $c=\alpha_2=s$,  and
$\alpha_0$, $\alpha_1$, and $\alpha_2=s$ are the zeros of the third degree polynomial 
\begin{equation}
P_{3}(t)P_{2}(s) - P_{3}(s)P_{2}(t)=0, \label{op_eq1}
\end{equation}
where $P_i(t)=P_i^{(\frac{n-1}{2},\frac{n-3}{2})}(t)$ is a Jacobi polynomial normalized for $P_i(1)=1$,
and $s$ is determined as the maximal root of the equation $M=L_5(n,s)$ (see Theorem 4.1 in \cite{Lev92}).

We present new proof of the above facts by using a slight generalization of a result from \cite{BN} on the structure of spherical designs.  
We use the following definition for a spherical design \cite{FL}: a code
$C \subset \mathbb{S}^{n-1}$ is a spherical $\tau$-design if and
only if for any point $y \in \mathbb{S}^{n-1}$ and any real
polynomial $f(t)$ of degree at most $\tau$, the equality
\begin{equation}
\label{defin_f}
\sum_{x \in C}f(\langle x,y \rangle ) = f_0|C|
\end{equation}
holds, where $f_0$ is the first coefficient in the Gegenbauer expansion $f(t)=\sum_{i=0}^k f_i P_i^{(n)}(t)$. Another useful tool 
is the Levenshtein quadrature \cite[Theorems 4.1, 4.2]{Lev92}, a particular case of which says that 
\begin{equation}
\label{qf}
f_0=\frac{f(1)}{M}+\rho_0 f(\alpha_0)+\rho_1 f(\alpha_1)+\rho_2 f(\alpha_2)
\end{equation}
for every polynomial $f$ of degree at most 5. Here the weights $\rho_i$, $i=0,1,2$, are positive.

\begin{theorem} \label{max-l}
If $C \subset \mathbb{S}^{n-1}$ is a spherical 3-distance 5-design, then its inner products are exactly as described above. 
\end{theorem}

\begin{proof}
Using \eqref{qf} and suitable polynomials in \eqref{defin_f} Boyvalenkov and Nikova \cite[Corollary 2.3]{BN} proved a general result which 
implies in our context that all four intervals $[-1,\alpha_0]$, $[\alpha_0,\alpha_1]$, $[\alpha_1,\alpha_2]$, and 
$[\alpha_2,1)$ contain an inner product of $C$. This means that at least one of these intervals does not contain an inner product as interior point, 
implying as in \cite{BN} that all intervals do so, i.e. $a=\alpha_0$, $b=\alpha_1$, and $c=\alpha_2$.
\end{proof}

\begin{remark}
The generalization of Corollary 2.3 from \cite{BN} mentioned above consists of the fact that it is valid not only for 
$s \in (\xi_{k-1},\eta_k)$ as stated in \cite{BN} (we use here the notations from that paper) but for every $s \in (\xi_k,t_k)$, where $t_k>\eta_k$ 
is the largest zero of the Gegenbauer polynomial $P_k^{(n)}(t)$ (see also the comment after Theorem 2.3 in \cite{BBD}).
In fact, this observation allows any cardinalities $M \geq 2$ (though we need only $M \leq n(n+1)(n+5)/6$). 
\end{remark}

The proof of Theorem \ref{max-l} is naturally (and obviously) extended to the the following general result. 

\begin{theorem} \label{max-l-general}
If $C \subset \mathbb{S}^{n-1}$ is a spherical $k$-distance $(2k-1)$-design, then its inner products are exactly the roots 
of the Levenshtein polynomial
\begin{equation}
P_{k}(t)P_{k-1}(s) - P_{k}(s)P_{k-1}(t)=0, \label{op_eq2}
\end{equation}
where $P_i(t)=P_i^{(\frac{n-1}{2},\frac{n-3}{2})}(t)$ is a Jacobi polynomial normalized for $P_i(1)=1$ and $s$ is
determined as the maximal root of the equation
\[ |C|=L_{2k-1}(n,s), \]
$L_{2k-1}(n,s)$ is the Levenshtein bound.
\end{theorem}

\begin{proof}
Similarly to above, if follows from \cite[Corollary 2.3]{BN} that each of the intervals $[-1,\alpha_0]$, $[\alpha_0,\alpha_1]$, \ldots, $[\alpha_{k-1},1)$
contains an inner product of $C$. The pigeonhole principle now gives that one of these intervals does not have an inner product of $C$
as its interior point. Therefore, again by \cite[Corollary 2.3]{BN}, we conclude that the inner products are exactly
$\alpha_0,\alpha_1,\ldots,\alpha_{k-1}$ and $|C|=L_{2k-1}(n,s)$.
\end{proof}

\section{Distance distributions}

Denote by
\[ A_t(x):=|\{y \in C: \langle x,y \rangle \}| \]
the number of the points of $C$ having inner product $t$ with $x$. Then the system of 
numbers $(A_t(x): t  \in [-1,1))$ is called distance distribution of $C$ with respect of $x$. 

For $x \in C$ let $(X,Y,Z)=(A_a(x),A_b(x),A_c(x))$ be the distance distribution of $C$ with respect of $x$. 
Applying the quadrature formula \eqref{qf} with the polynomials $f(t)=t^{i}$, $i=0,1,\ldots,5$, we obtain that 
the seven numbers $a$, $b$, $c$, $X$, $Y$, $Z$ and $M=|C|$ satisfy the following system of six equations
\begin{equation} \label{syst1}
a^iX+b^iY+c^iZ=f_i M-1, \ \ i=0,1,\ldots,5, 
\end{equation}
where $f_i=0$ for odd $i$, $f_0=1$, $f_2=1/n$ and $f_4=3/n(n+2)$.

Taking the equations with odd $i$
\[ aX+bY+cZ=a^3X+b^3Y+c^3Z=a^5X+b^5Y+c^5Z=-1. \]
and assuming that $a+b$, $b+c$, $c+a$ are all nonzero, we resolve with respect to $X$, $Y$ and $Z$ 
(see \cite[Theorem 3.4]{BDL}; this is a Vandermonde-like system) to obtain
\begin{eqnarray*}
X &=& -\frac{(1-b^2)(1-c^2)}{a(a^2-b^2)(a^2-c^2)}, \\
Y &=& -\frac{(1-c^2)(1-a^2)}{b(b^2-c^2)(b^2-a^2)}, \\
Z &=& -\frac{(1-a^2)(1-b^2)}{c(c^2-a^2)(c^2-b^2)}. \\
\end{eqnarray*}
This implies, in particular, that the distance distribution of $C$ does not depend on the choice of the point $x$ and we therefore will omit $x$ 
in what follows.  

The computation of the distance distribution of the codes in Conjecture \ref{cb} (3) gives 
\begin{eqnarray*}
X &=& \frac{m(m^2-2)(m-1)^3}{4}, \\
Y &=& (m^2-1)^3, \\
Z &=& \frac{m(m^2-2)(m+1)^3}{4}. \\
\end{eqnarray*}

Further algebraic manipulations with the system \eqref{syst1} lead to another proof of Theorem \ref{max-l}. In particular, we
obtain that $a$, $b$, and $c$ are the roots of the equation 
\begin{equation} \label{GrindEQ__5}
(n+2)[n(n+3)-2M]t^3-n(n+2)(n-1)t^2+(6M-5n^2-7n)t+n(n-1)=0
\end{equation}
(compare to \eqref{abc-equ}). We use \eqref{GrindEQ__5} to compute 
the elementary symmetric polynomials $a+b+c$, $ab+bc+ca$, and $abc$ as rational functions of $n$ and $M$. 
It is clear that further advances in this direction can be based only on the investigation of the integrality conditions for the distance 
distribution of $C$. We focus on this in the next section. 

\section{Investigation of $XYZ$}

Since $X,Y,Z$ are not symmetric in $a$, $b$, $c$, we need to produce and consider symmetric expressions 
in order to express then as functions of $n$ and $M$. We present the investigation of the product $XYZ$.  It follows from the 
explicit formulas from the previous section that 
\begin{equation}
\label{GrindEQ__15}
XYZ=\frac{1}{abc} \left(\frac{(1-a^2)(1-b^2)(1-c^2)}{(a^2-b^2)(b^2-c^2)(c^2-a^2)}\right)^2.
\end{equation}  

Using Vieta formulas and elementary symmetric polynomials from \eqref{GrindEQ__5} we consecutively 
compute
\[ \frac{(1-a^2)(1-b^2)(1-c^2)}{(a+b)(b+c)(c+a)}=\frac{M(n-1)}{n(n+2)},\] 
\[ abc=\frac{n(n-1)}{(n+2)(2M-n(n+3))},\]
and (the most complicated) 
\[ (a-b)^2(b-c)^2(c-a)^2=\frac{R_1(n,M)}{(n+2)^3(2M-n(n+3))^4},\] 
where
\begin{eqnarray*}
R_1(n,M) &=&1728M^4-5184M^3n^2-8640M^3n+360M^2n^5+5544M^2n^4 +18936M^2n^3 \\ 
&& +16632M^2n^2-528Mn^7 -3424Mn^6-13056Mn^5-23712Mn^4-14576Mn^3 \\
&& +4n^{10}+ 178n^9+1086n^8+3428n^7+7856n^6+10218n^5+4878n^4.
\end{eqnarray*}
Therefore,
\begin{equation}
\label{GrindEQ__19}
XYZ=\frac{M^2\left(n-1\right)(n+2)^2{\left(2M-n\left(n+3\right)\right)}^5}{R_1(n,M)n^3}.
\end{equation}  

We proceed with derivation of divisibility conditions from \eqref{GrindEQ__19}. As usually, we denote by $v_p(A)$ the 
largest power of the prime number $p$ which divides an integer $A$.

\begin{theorem} \label{lem-n-2M}
If $C \subset \mathbb{S}^{n-1}$ is a spherical 3-distance 5-design, then $n$ divides $2M$. 
\end{theorem}

\begin{proof} The cases $n=2$ and $M=n(n+1)$ (i.e. $C$ is a tight spherical 5-design) are trivial. So we may continue in the above context. 

First we see that $n$ divides $M^2(2M)^5$. Hence, all odd primes divisors of $n$ divide $2M$ as well. 
For any odd prime divisor $p$ of $n$ we assume that $v_p(M)<v_p(n)$ and compute the power of $p$ in the numerator as 
\[ v_p(M^2{\left(2M-n\left(n+3\right)\right)}^5)=7v_p\left(M\right) \] 
and the denominator as 
\[ v_p (n^3R_1(n,M))=3v_p\left(n\right)+4v_p\left(M\right) \]
(the case $p=3$ needs separate consideration but leads to the same conclusion). 
Hence, $7v_p\left(M\right)\ge 3v_p\left(n\right)+4v_p\left(M\right)$, i.e. $v_p\left(M\right)\ge v_p\left(n\right)$, a contradiction. 

Further, denote $v_2\left(M\right)=y$, $v_2\left(n\right)=x$ and assume that $x>1+y\ge 2$ for a contradiction (the cases $y=0$ and 1 are easily
ruled out). Then the powers of 2 in the denominator and the numerator are
\[v_2 (n^3R_1(n,M))\ge 3x+v_2\left(4M^4\right)=3x+4y+2, \] 
\[ v_2(M^2\left(n-1\right)(n+2)^2{\left(2M-n\left(n+3\right)\right)}^5)=2y+5\left(y+1\right)+2=7y+7, \]
respectively. Therefore 
\[ 3y+5\ge 3x\ge 3\left(y+2\right)=3y+6, \] 
a contradiction. This completes the proof. \end{proof}

\begin{remark}
Note that the conclusion of Theorem \ref{lem-n-2M} holds true in the cases (1) and (2) of Conjecture \ref{cb}.
\end{remark}

Let $2M=Tn$ for some positive integer $T$. After cancelation, we arrive at
\begin{equation} \label{equ-xyz-T}
XYZ=\frac{T^2\left(n-1\right)(n+2)^2{\left(T-n-3\right)}^5}{4R_2(n,T)},
\end{equation}
where
\begin{eqnarray*}
R_2(n,T) &=& 108T^4-648T^3n^2-1080T^3n+90T^2n^5+1386T^2n^4+4734T^2n^3 \\
&& +4158T^2n^2-264Tn^7-1712Tn^6-6528Tn^5-11856Tn^4-7288Tn^3 \\
&& +2\left(n+1\right)\left(n+3\right)(2n^4+81n^3+213n^2+619n+813).
\end{eqnarray*}
We were unable to continue with significant divisibility conclusions from \eqref{equ-xyz-T}. 

In the end of this section we express either in terms of $m$ and the dimension $n$ 
the above parameters in the case (3) of Conjecture \ref{cb}. Taking the values of $X$, $Y$ and $Z$ from Section 3, we obtain 
\[ XYZ=\frac{m^2(m^2-2)^2(m^2-1)^3}{16}=\frac{(n+5)(n-1)^2(n+2)^3}{2^4.3^6}. \]
Theorem \ref{lem-n-2M} gives $T=m^4=(n+5)^2/9$.

\section{Derived codes}

We descibe a construction from \cite{DGS} defining derived codes (sections) which produces good codes in lower dimensions 
provided the original codes are good. 

Following \cite[Section 8]{DGS} we consider the three derived codes, $C_a$, $C_b$, and $C_c$,
of $C$ with respect to the three inner products, respectively. These codes are spherical 3-designs \cite[Theorem 8.2]{DGS}
and admit at most three
inner products. Therefore their distance distributions satisfy a system of 4 equations, which could be investigated.

For fixed point $x \in C$, the set
\[ C_a(x):=\{ y \in C: \langle x,y \rangle =a \} \]
defines, after rescaling on $\mathbb{S}^{n-2}$, a spherical 3-design which we denote by $C_a$. The cardinality of $C_a$ is obviously equal to $X$. 
The inner products of $C_a$ are found by the Cosine Law to be these among
\[ \frac{a-a^2}{1-a^2}=\frac{a}{1+a}, \ \ \frac{b-a^2}{1-a^2}, \ \ \frac{c-a^2}{1-a^2} \]
belonging to $[-1,1]$. 

The distance distribution of $C_a$ does not depend on the choice of the point and will be denoted
by $(X_a,Y_a,Z_a)$. It satisfies the system 
\begin{equation} \label{syst2}
\left(\frac{a}{1+a}\right)^i X_a + \left(\frac{b-a^2}{1-a^2}\right)^i Y_a+ \left(\frac{c-a^2}{1-a^2}\right)^i Z_a = f_iX-1, \ \ i=0,1,2,3,
\end{equation}
where $f_i$ are defined as in \eqref{syst1} (note that $n$ is replaced by $n-1$). 
The codes $C_b$ and $C_c$ are defined analogously. 

\section{Computer investigation of the distance distributions}

The conditions from Sections 2 and 3 allow easy computational investigation of particular cases of
Conjecture \ref{cb}. For fixed dimension $n$, we consider all feasible cardinalities 
\[ M \in \left(n(n+1), \frac{n(n+1)(n+5)}{6}\right], \]
satisfying the condition $2M=nT$ from Theorem \ref{lem-n-2M}. We compute $X$, $Y$ and $Z$ in each case
and check if they are nonnegative integers. It is clear that enough precision ensures the correctness of the results. 

We implemented this algorithm in all dimensions $n \leq 1000$. This computation confirms Conjecture \ref{cb} in all cases
except for the pairs
\[ (n,T)=(341,3744), \ (638,7011), \]
where analysis of the derived codes is needed. In both cases, the distance distribution of the derived codes gives a contradiction.

For $\left(n,T\right)=(341,\ 3744)$ we find from \eqref{GrindEQ__5} that $\left(a,b,c\right)=(-\frac{1}{7},-\frac{1}{35},\frac{1}{14})$. 
It follows that  $X=23205.$ Further, resolving $\eqref{syst2}$ for $i=0,1,2,$ we find
 $\left(X_a,Y_a,Z_a\right)=\left(\frac{1872}{7},\frac{571392}{49}, \frac{552500}{49}\right)$, a contradiction. 
Analogously, in the case $\left(n,T\right)=(638,\ 7011)$ we find $\left(a,b,c\right)=(-\frac{1}{8},-\frac{1}{40},\frac{1}{20})$, whence
$X=40508$ and, finally, $\left(X_a,Y_a,Z_a\right)=\left(\frac{52193}{224},\frac{1245375}{56}, \frac{577125}{32}\right)$, a contradiction.

 We note that these two 
exceptional solutions show that the proof can not be contibued by proving that all solutions satisfy $n+5|9T$,
as the expected solution suggests. 

Summarizing, we have verified the following theorem. 

\begin{theorem}
Conjecture \ref{cb} is true in all dimensions $n \leq 1000$. 
\end{theorem}

\section{Conclusions}

We have proved that all spherical 3-distance 5-designs have inner products exactly as coming from 
the Levenshtein framework \cite{Lev92,BD} and generalize this result. We prove that in all cases twice the cardinality 
$2M$ is divisible by the dimension $n$. Using a computer we have verified Conjecture \ref{cb} in all dimensions $n \leq 1000$.

\end{document}